\documentclass[12pt]{amsart}
\usepackage[lite]{amsrefs} 
\usepackage{amssymb,amsmath,enumitem}
\usepackage{graphicx}
\newcommand{\R}{\mathbb{R}}

\newcommand{\N}{\mathbb{N}}

\newcommand{\M}{\mathcal{M}}

\newcommand{\A}{\mathcal{A}}

\newcommand{\W}{{\bf{W}}}
\newcommand{\I}{{\bf{I}}}
\newcommand{\g}{{\bf{G}}}
\numberwithin{equation}{section}
\usepackage{hyperref}
\hypersetup{colorlinks=true, citecolor=blue, linkcolor=red, urlcolor=blue}
\theoremstyle{plain}
\newtheorem{Thm}{Theorem}[section]
\newtheorem{Cor}[Thm]{Corollary}
\newtheorem{Lem}[Thm]{Lemma}
\newtheorem{Prop}[Thm]{Proposition}

\theoremstyle{definition}

\newtheorem{Rem}[Thm]{Remark}

\theoremstyle{remark}

\begin{document}
\title[Solutions to nonlinear elliptic equations]
{Solutions in Lebesgue spaces to nonlinear elliptic equations 
with sub-natural growth terms}

\author{Adisak Seesanea}
\address{Department of Mathematics, Hokkaido University,  Sapporo, \newline Hokkaido 060-0810, Japan}
\email{\href{mailto:seesanea@math.sci.hokudai.ac.jp}{seesanea@math.sci.hokudai.ac.jp}}

\author{Igor E. Verbitsky}
\address{Department of Mathematics, University of Missouri,  Columbia,   \newline 
Missouri 65211, USA}
\email{\href{mailto:verbitskyi@missouri.edu}{verbitskyi@missouri.edu}}
\begin{abstract}
We study the existence problem for positive solutions $u \in L^{r}(\R^{n})$,
$0<r<\infty$, to the quasilinear elliptic equation
\[ 
-\Delta_{p} u  = \sigma u^{q}  \quad \text{in} \;\; \R^n
\] 
in the sub-natural growth case $0<q< p-1$, 
where $\Delta_{p}u = \text{div}( |\nabla u|^{p-2} \nabla u )$ is the $p$-Laplacian with 
$1<p<\infty$, and $\sigma$ is a nonnegative measurable function (or measure) on $\R^n$.

Our techniques rely on a study of general integral equations 
involving nonlinear potentials and related weighted norm inequalities. They are  applicable to more general quasilinear 
elliptic operators such as the  
$\A$-Laplacian $\text{div} \A(x,\nabla u)$, and the fractional Laplacian 
$(-\Delta)^{\alpha}$ on $\R^n$, as well as  linear
uniformly elliptic operators with bounded measurable coefficients $\text{div}(\A \nabla u)$ on an 
arbitrary domain $\Omega \subseteq \R^n$ with a positive Green function.
\end{abstract}
\thanks{A. S. is partially supported by JSPS Grant Number 17H01092}
\maketitle
\section{Introduction}\label{sect:intro} 
Let us consider the model quasilinear elliptic problem 
\begin{equation}\label{eq:p-laplacian}
\begin{cases}
-\Delta_{p} u  = \sigma u^{q}, \quad u\ge 0  \quad \text{in} \;\; \R^n, \\
\liminf\limits_{x\rightarrow \infty}u(x) = 0,
\end{cases}
\end{equation}
in the sub-natural growth case $0<q< p-1$.

Here $\Delta_{p}u = \text{div}( |\nabla u|^{p-2} \nabla u )$ is the $p$-Laplacian with 
$1<p<\infty$, and $\sigma$ is a nonnegative locally integrable function, or 
more generally, nonnegative Radon measure on $\R^{n}$ (in brief, $ \sigma \in \M^{+}(\R^n)$). 

In this paper, we use  the notion of $p$-superharmonic solutions, or equivalently, 
locally renormalized solutions. We refer to  \cite{HKM, KKT} for the corresponding definitions. Throughout, we tacitly assume that $u\in L^{q}_{loc}(\Omega, d\sigma)$, so that the right-hand side of \eqref{eq:p-laplacian} is a Radon measure. 

We establish a sharp sufficient condition on $\sigma$, in terms of integrability of  
nonlinear potentials, for the existence of a positive  solution $u \in L^{r}(\R^n)$, $0<r<\infty$, 
to problem \eqref{eq:p-laplacian}. 

The Wolff potential (or, more precisely, Havin-Maz'ya-Wolff potential, see \cite{HM, HW, KM}) 
$\W_{\alpha, p}\sigma$ of a measure 
$\sigma \in \mathcal{M}^{+}(\R^n)$ is defined, for $1<p<\infty$ and 
$0<\alpha<\frac{n}{p}$, by 
\[
\W_{\alpha, p}\sigma(x) = \int_{0}^{\infty} \left[ \frac{\sigma(B(x,r))}{r^{n-\alpha p}}\right]^{\frac{1}{p-1}}\; \frac{dr}{r}, \quad x \in \R^n
\]
where $B(x,r) = \{ y \in \R^n : |x-y|< r \}$ is a ball centered at $x \in \R^n$ of radius $r >0$.
Notice that $\W_{\alpha, p}\sigma=+\infty$ for $\alpha \ge \frac{n}{p}$ unless 
$\sigma = 0$.

In the linear case $p=2$, it reduces (up to a normalization constant) to the Riesz potential 
of order $2\alpha$, $\W_{\alpha, 2}\sigma = \I_{2\alpha}\sigma$, and in particular, 
$\W_{1, 2}\sigma =\I_{2}\sigma$ is the Newtonian potential. 
The reader may wish to consult, for example, \cite{AH, HKM, KM, MZ, Maz} for a comprehensive discussion
of nonlinear  potentials and their applications to partial differential equations.

We observe that bilateral pointwise estimates of nontrivial (minimal) solutions $u$ to 
\eqref{eq:p-laplacian} in the sub-natural growth case were obtained in \cite{CV2}:
\begin{equation} \label{two-sided}
c^{-1} [(\W_{1, p} \sigma)^{\frac{p-1}{p-1-q}}+\mathbf{K}_{1, p, q}  \sigma] 
\le u \le c [(\W_{1, p} \sigma)^{\frac{p-1}{p-1-q}} +\mathbf{K}_{1, p, q}  \sigma],  
\end{equation}
where $c>0$ is a constant which  depends only on $p$, $q$, and $n$. 

Here $\mathbf{K}_{1, p, q}$ is the so-called \textit{intrinsic} Wolff potential 
associated with \eqref{eq:p-laplacian} introduced in \cite{CV2}. Thus, a necessary 
and sufficient condition for the existence of a solution $u \in L^r(\R^n)$ 
to \eqref{eq:p-laplacian} is given by: 
\begin{equation} \label{cond-two-sided}
\W_{1, p}  \sigma \in L^{\frac{r(p-1)}{p-1-q}}(\R^n) \quad \textrm{and} \quad \mathbf{K}_{1, p, q}  \sigma \in L^r(\R^n).
\end{equation}

However, $\mathbf{K}_{1, p, q} \sigma$ is defined in terms of localized embedding 
constants in certain weighted norm inequalities, which  makes the second condition  in \eqref{cond-two-sided} quite complicated, and difficult to verify. 

In this paper, our goal is to give convenient sufficient conditions for the existence of 
solutions $u \in L^r(\R^n)$ in terms of potentials $\W_{1, p} \sigma$. These conditions yield simpler, but more restrictive sufficient conditions 
 of the type $\sigma \in L^s(\R^n)$, which are known to be sharp in the scale of Lebesgue spaces (see \cite{BO, CV3}).  

Our approach is also applicable to  similar existence results for the fractional Laplace problem
\begin{equation} \label{eq:frac-laplacian}
\begin{cases}
\left(-\Delta \right)^{\alpha} u = \sigma u^{q}, \quad u \ge 0  \quad \text{in} \;\; \mathbb{R}^n, \\
\liminf\limits_{x \rightarrow \infty}u(x) = 0,
\end{cases}
\end{equation}
where $0<q<1$ and $\left(-\Delta \right)^{\alpha}$ is the fractional
Laplacian with $0< \alpha < \frac{n}{2}$.

In the classical case $\alpha = 1$ (or equivalently, $p=2$)
our methods are adapted to establish a sufficient condition for the existence of a positive
superharmonic solution $u \in L^{r}(\Omega)$, 
$0<r<\infty$, to the problem
\begin{equation} \label{eq:laplacian}
\begin{cases}
-\Delta u = \sigma u^{q},  \quad u \ge 0 \quad \text{in} \;\; \Omega, \\
\liminf\limits_{x \rightarrow y} u(x) = 0, \quad y \in \partial\Omega,
\end{cases}
\end{equation}
where $0<q<1$ and $\Omega \subseteq \R^n$ is an arbitrary domain (possibly unbounded)
which possesses a positive Green function. 

We observe that the existence and uniqueness of {\it bounded} solutions to \eqref{eq:laplacian} on 
$\Omega = \R^n$ were established by Brezis and Kamin in \cite{BK}. 

In \cite{CV1} (see also \cite{SV1}), it was proved that the condition
\begin{equation}\label{cond:dsigma-wolff} 
\W_{1,p}\sigma \in L^{\frac{(\gamma+q)(p-1)}{p-1-q}}(\R^n, d\sigma)
\end{equation}
with $\gamma =1$, is necessary and sufficient for the existence of a positive {\it finite energy} solution $u \in \dot{W}_{0}^{1,p}(\R^n)$ to \eqref{eq:p-laplacian}. Moreover, such a solution 
is unique in $\dot{W}_{0}^{1,p}(\R^n)$.

Here, for $1 \leq p<\infty$ and a nonempty open set $\Omega \subseteq \R^n$, 
by $\dot{W}^{1,p}_{0}(\Omega)$ we denote the homogeneous Sobolev (or Dirichlet) 
space defined \cite{HKM} as the closure of $C_{0}^{\infty}(\Omega)$ with respect to the (semi) norm 
\[
\| u \|_{\dot{W}^{1,p}_{0}(\Omega)} =  \| \nabla u \|_{L^{p}(\Omega)}.
\]

In this  study, we will show that for a given $0<r<\infty$, 
condition \eqref{cond:dsigma-wolff}
with a suitable value of $\gamma=\gamma(r)>0$ yields the existence of a positive
$p$-superharmonic solution $u \in L^{r}(\R^n)$ to \eqref{eq:p-laplacian}.

Our main results are new even in the case $p=2$, or when $\sigma$ is a nonnegative locally integrable function on $\R^n$.

\begin{Thm}\label{thm:main-p-laplacian}  
Let $1<p<n$, $0<q<p-1$, $\frac{n(p-1)}{n-p}<r<\infty$, and $\sigma \in \mathcal{M}^{+}(\R^n)$ with $\sigma \not\equiv 0$. If \eqref{cond:dsigma-wolff} holds with 
$\gamma = \frac{r(n-p)-n(p-1)}{n}$, then there exists a positive $p$-superharmonic solution $u \in L^r(\R^{n})$ to \eqref{eq:p-laplacian}. 

If $n \leq p < \infty$, or $1<p<n$ 
and $0<r\le \frac{n(p-1)}{n-p}$, then 
there is only a trivial supersolution to \eqref{eq:p-laplacian}.
\end{Thm}

A necessary  (but generally not sufficient) condition 
for the existence of a nontrivial solution $u \in L^{r}(\R^n)$ to \eqref{eq:p-laplacian}, 
where $1<p<n$ and $\frac{n(p-1)}{n-p}<r<\infty$,  follows from \eqref{cond-two-sided} (see also Theorem \ref{thm:lower-est-p} below),  
\begin{equation}\label{cond:dx-wolff-1} 
\W_{1,p} \sigma \in L^{\frac{r(p-1)}{p-1-q}}(\R^n). 
\end{equation}
Here we use $L^s(\R^n)$  with respect to Lebesgue measure, whereas in condition \eqref{cond:dsigma-wolff}, $L^s (\R^n, d \sigma)$ is with respect to the measure 
$\sigma$. 

A simple sufficient condition for \eqref{cond:dsigma-wolff} with $\gamma = 
\frac{r(n-p)-n(p-1)}{n}$
is given by
\begin{equation}\label{cond:sigma-suff}
\sigma \in L^{s_1}(\R^n),   \quad s_{1}=\tfrac{nr}{n(p-1-q)+ p r},  
\end{equation}
where  $1<p<n$ and $\frac{n(p-1)}{n-p}<r<\infty$ (see Proposition \ref{prop:wolff} with $\alpha=1$).
Thus, in light of Theorem \ref{thm:main-p-laplacian}, we have the following result.

\begin{Cor}\label{cor:main1}
Under the assumptions of Theorem \ref{thm:main-p-laplacian}, if 
 \eqref{cond:sigma-suff} holds, then there exists a positive $p$-superharmonic solution 
$u \in L^r(\R^n)$ to \eqref{eq:p-laplacian}, provided $1<p<n$ and $\frac{n(p-1)}{n-p}<r<\infty$.
\end{Cor}

\begin{Rem}\label{rem-main1}
 The sufficient conditions of Theorem \ref{thm:main-p-laplacian} and Corollary 
\ref{cor:main1}, along with the necessary condition \eqref{cond:dx-wolff-1}, hold for the $\A$-Laplacian in place of $\Delta_{p}$, under the standard structural  assumptions on $\mathcal{A}$ (see \cite{CV2, HKM, MZ}).
\end{Rem}

The main new feature in our approach is the use of 
two weight  inequalities involving Wolff potentials (see Lemma 
\ref{lem1-wolff}  and Lemma \ref{lemma-link-p} below).

Our methods in the case $p=2$ are applicable to analogous existence results for the fractional Laplace problem 
\eqref{eq:frac-laplacian}.

\begin{Thm}\label{thm:main2}
Let $0<q<1$, $0<\alpha<\frac{n}{2}$ and $\sigma \in \mathcal{M}^{+}(\R^n)$ with 
$\sigma \not\equiv 0$. Suppose that $\frac{n}{n-2\alpha}<r<\infty$ and
\begin{equation}\label{cond:dsigma-riesz} 
\I_{2\alpha}\sigma \in L^{\frac{\gamma+q}{1-q}}(\R^n, d\sigma), 
\end{equation} 
where $\gamma = \frac{r(n-2\alpha)-n}{n}$. Then there exists a positive  solution $u \in L^r(\R^{n})$ to \eqref{eq:frac-laplacian}.

If $0<r\le \frac{n}{n-2\alpha}$, then 
there is only a trivial supersolution to \eqref{eq:frac-laplacian}.
\end{Thm}

Here, a solution $u\ge 0$ to problem \eqref{eq:frac-laplacian} is understood in the sense that
$u \in L^{q}_{loc}(\R^n, \sigma)$ satisfies the corresponding 
integral equation
\begin{equation}\label{eq:int2}
u = \I_{2\alpha}(u^{q} d\sigma) \quad \text{in} \;\; \R^{n}.
\end{equation}

A necessary condition for the existence of such a solution $u \in L^{r}(\R^n)$,  to \eqref{eq:frac-laplacian}, which follows from Theorem \ref{lower_ptwise_est} below with $p=2$, is given by
\begin{equation}\label{cond:dx-riesz-1} 
\I_{2\alpha} \sigma \in L^{\frac{r}{1-q}}(\R^n).
\end{equation} 

As above, a simple sufficient condition for \eqref{cond:dsigma-riesz} with $\gamma = \frac{r(n-2\alpha)-n}{n}$ 
is provided by
\begin{equation}\label{cond:sigma-suff2}
\sigma \in L^{s_2}(\R^n), \quad  s_{2}=\tfrac{nr}{n(1-q)+2\alpha  r}, 
\end{equation}
where $\frac{n}{n-2\alpha}<r<\infty$.

\begin{Cor}\label{cor:main2}
Under the assumptions of Theorem \ref{thm:main2},  
if \eqref{cond:sigma-suff2} is valid, then there exists a positive solution 
$u \in L^r(\R^n)$ to \eqref{eq:frac-laplacian}.
\end{Cor}

We observe that  condition \eqref{cond:dsigma-riesz}
with $\gamma=1$ is necessary and sufficient for the existence of a positive finite energy solution
$u \in \dot{H}^{\alpha}(\R^n)$, $0<\alpha <\frac{n}{2}$, to \eqref{eq:frac-laplacian};  uniqueness of such a solution was proved in the case $0<\alpha\leq1$, see \cite{SV1}.

Here, the homogeneous Sobolev space $\dot{H}^{\alpha}(\R^n)$ ($0<\alpha<\frac{n}{2}$) can be defined by means of Riesz potentials, 
\[
\dot{H}^{\alpha}(\R^n) =  \big\lbrace u\colon \,  u = \I_{\alpha}f,\; f \in L^{2}(\R^n) \big\rbrace,
\]
equipped with  norm  $\|u \|_{\dot{H}^{\alpha}(\R^n)}= \| f \|_{L^{2}(\R^n)}$.

Finally, we consider the sublinear elliptic problem \eqref{eq:laplacian} on an arbitrary domain 
$\Omega \subseteq \R^n$ (possibly unbounded) with positive Green function 
$G(x,y)$ on $\Omega \times \Omega$. 
The corresponding Green potential of a measure $\sigma \in \M^{+}(\Omega)$ is defined by 
\[
\g\sigma(x) = \int_{\Omega} G(x,y) \, d\sigma(y), \quad x \in \Omega.
\]
Our main results for this problem read as follows.

\begin{Thm}\label{thm:main3}  
Let $0<q<1$, $\sigma \in \mathcal{M}^{+}(\Omega)$ with 
$\sigma \not\equiv 0$, and let $G$ be the positive Green function 
associated with $-\Delta$ in $\Omega \subseteq \R^n$, $n \ge 3$.
Suppose $\frac{n}{n-2}<r<\infty$ and
\begin{equation}\label{cond:dsigma-green} 
\g\sigma \in L^{\frac{\gamma+q}{1-q}}(\Omega, d\sigma), 
\end{equation} 
where $\gamma = \frac{r(n-2)-n}{n}$. Then there exists a positive superharmonic solution 
$u \in L^r(\Omega)$ to \eqref{eq:laplacian}.
\end{Thm}

Invoking Theorem \ref{Thm:lowerbound} below, in the special case where $G$ is a positive Green function associated with $-\Delta$, yields a necessary condition for the existence of such a solution,
\begin{equation}\label{cond:dx-green} 
\g \sigma \in L^{\frac{r}{1-q}}(\Omega).
\end{equation} 

A sufficient condition for \eqref{cond:dsigma-green} with $\gamma = \frac{r(n-2)-n}{n}$  is given by
\begin{equation}\label{cond:sigma-suff3}
\sigma \in L^{s_3}(\Omega), \quad  s_{3}=\tfrac{n r}{n(1-q)+2r}, 
\end{equation}
where $\frac{n}{n-2}<r<\infty$
(see Proposition \ref{prop3}, when $G$ is a positive Green function associated with 
$-\Delta$).
Thus, the following corollary is easily deduced from Theorem \ref{thm:main3}.

\begin{Cor}\label{cor:main3}
Under the assumptions of Theorem \ref{thm:main3},  
if condition \eqref{cond:sigma-suff3} holds, then there exists a positive superharmonic solution 
$u \in L^r(\Omega)$ to \eqref{eq:laplacian}.
\end{Cor}

Corollary \ref{cor:main3}, when $\Omega \subseteq \R^n$ is a  bounded domain, is due to Boccardo and Orsina \cite{BO}, with a different proof.

It was recently shown by the authors \cite{SV2} that, for an arbitrary $0<\gamma<\infty$, condition
\eqref{cond:dsigma-green} is necessary and sufficient for the existence a positive
superharmonic solution $u$ to \eqref{eq:laplacian} with {\it finite generalized energy}:
\[
\mathcal{E}_{\gamma}[u] = \int_{\Omega} |\nabla u|^{2} u^{\gamma - 1}\,dx < +\infty.
\]
Theorem \ref{thm:main3}  can be deduced from this result, but we 
give an independent proof below. 

\begin{Rem}\label{rem3}
The above results hold, with the same proofs,  for  linear
uniformly elliptic operators with bounded measurable coefficients 
$L=-\text{div}(\A \nabla \cdot)$ in place of $-\Delta$, on an 
arbitrary domain $\Omega \subseteq \R^n$ with positive Green function, as well as 
more general (nonlocal) integral equations $u= \mathbf{G}(u^q d\sigma)$ 
with positive kernel $G$. 

The restrictions on $G$ imposed below are the {\it weak maximum principle} (WMP), and {\it quasi-symmetry} 
(see Sec. \ref{sect:kernel}). We also sometimes add  the condition that $G$ is bounded above by the Riesz kernel 
$I_{2\alpha}(x-y)= |x-y|^{2\alpha-n}$
of order $2\alpha$ for some $0<\alpha<\frac{n}{2}$,
\begin{equation}\label{cond:Riesz0} 
G(x, y) \leq c I_{2\alpha}(x-y), \quad  \forall x, y \in \Omega,
\end{equation} 
where $c$ is a positive constant independent of $x$ and $y$.
\end{Rem}

This paper is organized as follows. 
In Sec. 2, we recall the necessary background and state preliminary results 
concerning nonlinear potentials and related weighted norm inequalities.
In Sec. 3, we simultaneously prove the existence results for positive solutions  
$u \in L^{r}(\R^n)$
to both problems \eqref{eq:p-laplacian} and \eqref{eq:frac-laplacian}.
Our main results for problem \eqref{eq:laplacian} will be established in Sec. 4.

\section{Preliminaries}
\subsection{Wolff potentials}\label{sect:wolff}
Recall that for $1<p<\infty$ and $0<\alpha<\frac{n}{p}$,
the (homogeneous) potential $\W_{\alpha, p}\sigma$ of a measure $\sigma \in \mathcal{M}^{+}(\R^n)$  is defined by \cite{HM, HW}
\[
\W_{\alpha, p}\sigma(x) := \int_{0}^{\infty} \left[ \frac{\sigma \left( B(x, r) \right)}{r^{n-\alpha p}}\right]^{\frac{1}{p-1}}\; \frac{dr}{r}, \quad x \in \R^n,
\]
where $B(x,r):= \{ y \in \R^n : |x-y|< r \}$ is a ball centered at $x \in \R^n$ of radius 
$r >0$.

It is well-known that
\begin{equation}\label{inq:HM}
\W_{\alpha, p}\sigma(x) \leq c \mathbf{V}_{\alpha, p}\sigma(x), \quad x \in \R^n, 
\end{equation}
where $\mathbf{V}_{\alpha, p}\sigma := \I_{\alpha}[(\I_{\alpha}\sigma)^{\frac{1}{p-1}}dx]$ is 
the Havin-Maz'ya potential, and $c=c(\alpha, p, n) $ is a positive constant.

Havin-Maz'ya potentials satisfy the weak maximum (or boundedness) 
principle (see \cite[Theorem 2.6.3]{AH}). A similar weak maximum principle holds 
for Wolff potentials: if $\sigma \in \mathcal{M}^{+}(\R^n)$, then
\begin{equation}\label{wolff-max}
\W_{\alpha, p}\sigma(x) \leq c  \, \sup \left \{\W_{\alpha, p}\sigma(y): \, y \in \textrm{supp}(\sigma)\right\}, \quad \forall x \in \R^n,
\end{equation}
where  $c=c(\alpha, p, n)$ is a positive constant.

The following pointwise iterated inequality for Wolff potentials will be used below.

\begin{Thm}[\cite{CV1}]\label{Thm:iterated-wolff}
Let $1<p< \infty$, $0<\alpha< \frac{n}{p}$, and $\sigma \in \M^{+}(\R^{n})$ 
with $\sigma \not\equiv 0$. Then the following estimates hold: 
\begin{itemize}
	\item[(i)]If $t\geq1$, then   
	\begin{equation}\label{iterated1-wolff}
(\W_{\alpha, p}\sigma)^{t}(x) \leq c \, \W_{\alpha, p} \left( (\W_{\alpha, p}\sigma)^{(t-1)(p-1)} d\sigma \right)(x), \quad  \forall x \in \R^{n}. 
	\end{equation}
	\item[(ii)] If $0<t\leq1$, then
	\begin{equation}\label{iterated2-wolff}
(\W_{\alpha, p}\sigma)^{t}(x) \geq c \, \W_{\alpha, p} \left( (\W_{\alpha, p}\sigma)^{(t-1)(p-1)} d\sigma \right)(x), \quad  \forall x \in \R^{n}.
	\end{equation}
\end{itemize}
Here, $c= c(t, p, n, \alpha) > 0$.
\end{Thm}

The following weak continuity result will be used to prove the existence of 
$p$-superharmonic solutions to quasilinear equations. We denote by 
$\omega[u]$ the $p$-measure associated with a $p$-superharmonic function $u$ (see \cite{HKM}). 

\begin{Thm} [\cite{TW}] \label{weak_cont_p-Laplacian}
Suppose $\lbrace u_{j} \rbrace_{j=1}^{\infty}$ is a sequence of nonnegative $p$-superharmonic functions in $\Omega$ such that $u_j \rightarrow u$ a.e. as $j \rightarrow \infty$, where $u$ is a $p$-superharmonic function in $\Omega$. Then $\omega[u_{j}]$ converges weakly to $\omega[u]$, that is,
\[
\lim_{j \rightarrow \infty} \int_{\Omega} \varphi \; d\omega[u_j] 
= \int_{\Omega} \varphi \; d\omega[u]
\]
for all $\varphi \in C_{0}^{\infty}(\Omega)$.
\end{Thm}

We shall use the following lower bounds for supersolutions to \eqref{eq:p-laplacian}.

\begin{Thm}[\cite{CV1}] \label{thm:lower-est-p} 
Let $1<p<n$, $0<q<p-1$ and $\sigma \in \mathcal{M}^{+}(\mathbb{R}^n)$.
Suppose $u$ is a nontrivial supersolution to \eqref{eq:p-laplacian}. Then $u$ satisfies the inequality 
\[
u \geq c \left( \W_{1, p}\sigma \right)^{\frac{p-1}{p-1-q}}, 
\]
where $c = c(n, p, q) >0$. If $p \geq n$ there is only a trivial supersolution $u=0$ on $\R^n$.
\end{Thm}

\begin{Thm} [\cite{CV2}] \label{lower_ptwise_est}  
Let $1<p<n$, $0<q<p-1$, $0<\alpha<\frac{n}{p}$ and $\sigma \in \mathcal{M}^{+}(\mathbb{R}^n)$. Suppose $u \in L^{q}_{loc}(\mathbb{R}^n, d\sigma)$, $u\ge 0$,  is a nontrivial  
function satisfying
\[
u \geq  \W_{\alpha, p} (u^{q}d\sigma) \quad \text{in} \;\; \R^n.
\]
Then $u$ satisfies the inequality
\begin{equation}\label{inq:lowerbound}
u \geq c \left( \W_{\alpha, p}\sigma \right)^{\frac{p-1}{p-1-q}} \quad \text{in} \;\; \R^n,
\end{equation}
where $c = c(\alpha, n, p, q)$ is a positive constant.
\end{Thm}

The following important result is concerned with pointwise estimate of nonnegative $p$-superharmonic functions in terms of Wolff's potential.

\begin{Thm} [\cite{KM}]\label{pointwise_est_p-superharmonic} 
Let $1<p<n$ and $\sigma \in \mathcal{M}^{+}(\mathbb{R}^n)$ Suppose 
$u$ is a $p$-superharmonic function in $\mathbb{R}^n$ 
satisfying
\[
\begin{cases}
-\Delta_{p} u = \sigma \quad \text{in} \;\; \mathbb{R}^n, \\ 
\liminf\limits_{\vert x \vert \rightarrow \infty} u(x) = 0.
\end{cases}
\]
Then 
\[
K^{-1}\W_{1,p} \sigma \leq u \leq K \W_{1,p}\sigma,
\]
where $K=K(n,p) \geq 1$.
\end{Thm}
\subsection{Kernels and potentials}\label{sect:kernel}
Let $G: \Omega \times \Omega \rightarrow (0, \infty]$ be a positive lower semicontinuous kernel. The potential of a measure $\sigma \in \mathcal{M}^{+}(\Omega)$ is defined by 
\[
\g \sigma (x) := \int_{\Omega} G(x,y) \;d\sigma (y), \quad x \in \Omega.
\]

Note that $\g\sigma(x)$ is lower semi-continuous on $\Omega \times \M^{+}(\Omega)$ if and only if $G(x,y)$ is lower semi-continuous on $\Omega \times \Omega$ (see \cite{Br}).

A positive kernel $G$ on $\Omega \times \Omega$ is said to satisfy the weak maximum principle (WMP) with constant $h \geq 1$ if for any $\sigma \in \mathcal{M}^{+}(\Omega)$,
\begin{equation}\label{WMP}
\sup \{ \g\sigma(x) : x \in \text{supp}(\sigma) \} \leq 1
\Longrightarrow
\sup \{ \g\sigma(x) : x \in \Omega \} \leq h.
\end{equation}
Here we use the notation supp$(\sigma)$ for the support of $\sigma$.

When $h=1$ in \eqref{WMP}, the kernel $G$ is said to satisfy the strong maximum principle.  It holds for Green functions associated with the classical Laplacian $-\Delta$, or more 
generally the linear uniformly elliptic operator in divergence form $L$, as well as 
the fractional Laplacian $(-\Delta)^{\alpha}$ in the case $0<\alpha\leq 1$, in every domain 
$\Omega \subset \R^n$ which possesses a Green function. 

The WMP holds for Riesz kernels on $\R^n$ associated with $(-\Delta)^{\alpha}$ in the full range 
$0< \alpha < \frac{n}{2}$, and more generally for all radially nonincreasing kernels on $\R^n$, 
see \cite{AH}.

We say that a positive kernel $G$ on $\Omega \times \Omega$ is quasi-symmetric if there exists 
a constant $a\geq 1$ such that 
\begin{equation}\label{quasi-sym}
a^{-1}G(y,x) \leq G(x,y) \leq a G(y,x), \quad x,y \in \Omega.
\end{equation}
When $a=1$ in \eqref{quasi-sym}, the kernel $G$ is said to be symmetric. 
There are many kernels associated with elliptic operators that are quasi-symmetric and 
satisfy the WMP, see \cite{An}.

We begin with the following pointwise estimates (see \cite[Theorem 1.3]{GV}) for supersolutions 
to sublinear integral equations such that
\begin{equation}\label{eq:intG}
u (x) \ge  \g(u^{q} d\sigma) (x) \quad \forall x \in \Omega.
\end{equation}
\begin{Thm}[\cite{GV}]\label{Thm:lowerbound}
Let $0<q<1$, $\sigma \in \mathcal{M}^{+}(\Omega)$, and let
$G$ be a positive lower semicontinuous kernel on 
$\Omega \times \Omega$, which satisfies the WMP with constant $h \geq 1$.
If $u \in L^{q}_{loc}(\Omega, d\sigma)$ is a positive supersolution satisfying \eqref{eq:intG}, then 
\begin{equation}\label{lowerbound}
u(x)  \geq (1-q)^{\frac{1}{1-q}} h^{-\frac{q}{1-q}} [ \g \sigma (x) ]^{\frac{1}{1-q}}, \quad  \forall x \in \Omega.
\end{equation}
\end{Thm}

The following iterated inequalities are used in our argument, see \cite[Lemma 2.5]{GV}.

\begin{Thm}[\cite{GV}]\label{Thm:iterated}
Let $\sigma \in \M^{+}(\Omega)$ with $\sigma \not\equiv 0$, and let $G$ be a positive lower semicontinuous kernel on 
$\Omega \times \Omega$, which satisfies the WMP with constant $h\geq1$. 
Then the following estimates hold: 
\begin{itemize}
	\item[(i)]If $t\geq1$, then   
	\begin{equation}\label{iterated1}
(\g\sigma)^{t}(x) \leq t h^{t-1} \, \g \left( (\g\sigma)^{t-1} d\sigma \right)(x), \quad  \forall x \in \Omega. 
	\end{equation}
	\item[(ii)] If $0<t\leq1$, then
	\begin{equation}\label{iterated2}
(\g\sigma)^{t}(x) \geq t h^{t-1} \, \g \left( (\g\sigma)^{t-1} d\sigma \right)(x), \quad  \forall x \in \Omega.
	\end{equation}
\end{itemize}
\end{Thm}

The next result   characterizes explicitly 
weighted norm inequalities of the $(s ,r)$-type  in the case $0< r < s$ and $1< s < \infty$,
 for 
operators $\g$ (see \cite[Theorem 1.1]{V}):
\begin{equation}\label{weighted_norm_ineq}
\big\| \g(f d\sigma) \big\|_{L^{r}(\Omega,\,d\sigma)} \leq c \| f \|_{L^{s}(\Omega,\,d\sigma)}, 
\quad \forall  f \in L^{s}(\Omega, d\sigma),
\end{equation}
where $c$ is a positive constant independent of $f$, for an arbitrary measure $\sigma \in \M^{+}(\Omega)$, under certain assumptions on $G$.

\begin{Thm}[\cite{V}]\label{thm_Ver17}
Let $\sigma \in \mathcal{M}^{+}(\Omega)$ with $\sigma \not\equiv 0$, and let $G$ be a positive,  
quasi-symmetric, lower semicontinuous kernel on $\Omega \times \Omega$, which satisfies the WMP.
\begin{itemize}
\item[(i)] If $1 < s < \infty$ and $0 < r < s$, then the weighted norm inequality 
\eqref{weighted_norm_ineq} holds if and only if 
\begin{equation}\label{energy1}
\g\sigma \in L^{\frac{sr}{s-r}}(\Omega, d\sigma).
\end{equation}

\item[(ii)] If $0 < q < 1$ and $0<\gamma<\infty$, then there exists a positive 
(super)solution 
$u \in L^{\gamma + q}(\Omega, d\sigma)$ to equation \eqref{eq:intG}
if and only if the weighted norm inequality \eqref{weighted_norm_ineq} holds with 
$r=\gamma +q$ and $s= \frac{\gamma +q}{q}$, i.e., 
\begin{equation}\label{weighted_norm_ineq_special}
\big\| \g(f d\sigma) \big\|_{L^{\gamma + q}(\Omega,\;d\sigma)} 
\leq c \| f \|_{L^{\frac{\gamma + q}{q}}(\Omega, d\sigma)}, 
\quad \forall f \in L^{\frac{\gamma + q}{q}}(\Omega,  d\sigma),
\end{equation}
or equivalently,
\begin{equation}\label{energy2}
\g\sigma \in L^{\frac{\gamma + q}{1-q}}(\Omega, d\sigma).
\end{equation}
\end{itemize}
\end{Thm}

The existence of solutions $u\in L^{q}(\Omega, d\sigma)$, and more generally 
$u\in L^{q}_{{\rm loc}}(\Omega, d\sigma)$, 
 is characterized in \cite{QV} via $(1,q)$-weighted norm inequalities, which corresponds to $\gamma=0$ in \eqref{weighted_norm_ineq_special}.   


\section{Solutions  to problems \eqref{eq:p-laplacian} and \eqref{eq:frac-laplacian}}
In this section, we prove the existence results stated in the introduction for positive solutions  $u\in L^{r}(\R^{n})$, $0<r< \infty$, to both problems \eqref{eq:p-laplacian} and \eqref{eq:frac-laplacian}.

When $p \geq n$, it readily follows from Theorem \ref{thm:lower-est-p} 
that there is only a trivial supersolution to \eqref{eq:p-laplacian}. Henceforth, we assume $1<p<n$. 

We begin with an investigation of solvability of the corresponding 
nonlinear integral equations involving Wolff potentials, 
\begin{equation}\label{eq:int-wolff}
u = {\W}_{\alpha,p} (u^{q}d\sigma) \quad \text{in} \;\; \mathbb{R}^n
\end{equation} 
where $1<p<n$, $0<q<p-1$, $0<\alpha<\frac{n}{p}$ and 
$\sigma \in \mathcal{M}^{+}(\mathbb{R}^n)$. 

The following theorem will be used in the construction of such solutions
to \eqref{eq:p-laplacian} and \eqref{eq:frac-laplacian} 
in the cases $\alpha = 1$ and $p=2$, respectively.

\begin{Thm}\label{thm:wolff-existence}
Let $0<\gamma<\infty$, $1<p<n$, $0<q<p-1$, $0<\alpha<\frac{n}{p}$ and $\sigma \in \mathcal{M}^{+}(\mathbb{R}^n)$ such that $\sigma \not\equiv 0$. Then there exists a positive solution  $u \in L^{\gamma+q}(\mathbb{R}^n, d\sigma)$ to
\eqref{eq:int-wolff} if and only if
 \begin{equation}\label{cond:sigma-gamma}
 {\W}_{\alpha, p}\sigma \in L^{\frac{(\gamma+q)(p-1)}{p-1-q}}(\R^n, d\sigma).
 \end{equation}
\end{Thm}

The proof of this theorem is based on the following lemma,  
which  is an extension of \cite[Lemma 3.3]{CV1} in the case $\gamma =1$.

\begin{Lem}\label{lem:wolff}
Let $0<\gamma<\infty$, $1<p<\infty$, $0<\alpha<\frac{n}{p}$, $0<q<p-1$, and $\sigma \in \M^{+}(\R^n)$. If \eqref{cond:sigma-gamma} holds then
\begin{equation}\label{winq:wolff-sigma-sigma}
\big\| \W_{\alpha, p}(f d\sigma) \big\|_{L^{\gamma + q}(\R^{n}, d\sigma)}
\leq c \| f \|_{L^{\frac{\gamma+q}{q}}(\R^{n}, d\sigma)}^{\frac{1}{p-1}}, 
\, \,  \forall f \in L^{\frac{\gamma+q}{q}}(\R^n, d\sigma),
\end{equation}
where $c$ is a positive constant independent of $f$.
\end{Lem}

\begin{proof}
Without loss of generality, we may assume $f \in L^{\frac{\gamma+q}{q}}(\R^n, d\sigma)$ 
with $f \geq 0$. Observe that
\[
\W_{\alpha, p}(f d\sigma) \leq \left( M_{\sigma}f \right)^{\frac{1}{p-1}} \W_{\alpha, p}\sigma 
\]
where the centered maximal operator $M_{\sigma}$ is defined by
\[
M_{\sigma}f (x) := \sup_{r>0} \frac{1}{\sigma(B(x,r))} \int_{B(x,r)} |f| \, d\sigma, \quad x \in \R^{n}.
\]
By H\"older's inequality with exponents  $\frac{p-1}{q}$ and $\frac{p-1}{p-1-q}$, and the 
boundedness of the maximal operator $M_{\sigma}: L^{s}(\R^{n}, d\sigma) \rightarrow L^{s}(\R^{n}, d\sigma) $ with $s=\frac{\gamma +q }{q}$, we obtain
\[
\begin{split}
& \int_{\R^n} \left(\W_{\alpha, p}(f d\sigma) \right)^{\gamma+q}\,d\sigma
 \leq \int_{\R^n} \left(M_{\sigma}f \right)^{\frac{\gamma+q}{p-1}} 
\left(\W_{\alpha, p}\sigma \right)^{\gamma+q}\,d\sigma \\
& \leq \left(  \int_{\R^n} \left(M_{\sigma}f \right)^{\frac{\gamma+q}{q}}\,d\sigma \right)^{\frac{q}{p-1}}
\left( \int_{\R^n} \left( \W_{\alpha, p}\sigma \right)^{\frac{(\gamma+q) (p-1)}{p-1-q}}\,d\sigma \right)^{\frac{p-1-q}{p-1}} \\
& \leq c \left( \int_{\R^n} f^{\frac{\gamma+q}{q}}\,d\sigma \right)^{\frac{q}{p-1}}
 \left( \int_{\R^n} \left(\W_{\alpha, p}\sigma \right)^{\frac{(\gamma+q)(p-1)}{p-1-q}}\,d\sigma\right)^{\frac{p-1-q}{p-1}}.
\end{split}
\]
Thus, by assumption \eqref{cond:sigma-gamma}, this yields \eqref{winq:wolff-sigma-sigma}.
\end{proof}

We now prove Theorem \ref{thm:wolff-existence}.

\begin{proof}[Proof of Theorem \ref{thm:wolff-existence}] To prove the sufficiency part, we construct a sequence of functions $\lbrace  u_j \rbrace_{j=0}^{\infty}$ by setting
\[
u_0 := c_{0}\left({\W}_{\alpha, p}\sigma\right)^{\frac{p-1}{p-1-q}} 
\quad \text{and} \quad 
u_{j+1} := {\W}_{\alpha, p} (u^{q}_{j} d\sigma), \quad \;\; j=0,1,2,\dots
\]
where $c_{0}$ is a small positive constant to be determined later.
Observe that $u_{0}>0$ since $\sigma \not\equiv 0$. Moreover, 
\[
u_{1} = {\W}_{\alpha, p} (u^{q}_{0} d\sigma)
= c_{0}^{\frac{q}{p-1}} {\W}_{\alpha, p} ( ({\W}_{\alpha, p}\sigma)^{\frac{q(p-1)}{p-1-q}} d\sigma)
\geq c_{0}^{\frac{q}{p-1}}C ({\W}_{\alpha, p}\sigma)^{\frac{p-1}{p-1-q}}
\]
where $C$ is the positive constant in inequality \eqref{iterated1-wolff} 
with $t=\frac{p-1}{p-1-q}$.
Choosing $c_{0}$ small enough so that $c_{0}^{\frac{q}{p-1}}C \geq c_{0}$, 
we obtain  $u_{0} \leq u_{1}$. 
By induction, we see that $\lbrace  u_j\rbrace_{j=0}^{\infty}$ is a nondecreasing sequence of positive functions. 

Next, we show that each $u_j \in  L^{\gamma+q}(\mathbb{R}^n, d\sigma)$. Note that assumption \eqref{cond:sigma-gamma} yields 
\[
u_0 = c_{0}\left({\W}_{\alpha, p}\sigma\right)^{\frac{p-1}{p-1-q}} \in L^{\gamma+q}(\mathbb{R}^n, d\sigma). 
\]

By Lemma \ref{lem:wolff},  weighted norm inequality \eqref{winq:wolff-sigma-sigma} holds.
Suppose that $u_0, \dots, u_j \in L^{\gamma +q}(\mathbb{R}^n, d\sigma)$ for some $j \in \N$.
Applying \eqref{winq:wolff-sigma-sigma} with $f:=u^{q}_{j} \in 
L^{\frac{\gamma+q}{q}}(\mathbb{R}^n, d\sigma)$, we deduce 
\begin{equation}\label{est2_thm_existence}
\begin{split}
 \Vert u_{j+1} \Vert_{L^{\gamma+q}(\mathbb{R}^n,\,d\sigma)}&  = \big\| \W_{\alpha, p} (u^{q}_{j} d\sigma) \big\|_{L^{\gamma+q}(\mathbb{R}^n,\,d\sigma)} \\
&\leq c \Vert u_{j}\Vert^{\frac{q}{p-1}}_{L^{\gamma+q}(\mathbb{R}^n,\,d\sigma)} \\
&\leq c \Vert u_{j+1}\Vert^{\frac{q}{p-1}}_{L^{\gamma+q}(\mathbb{R}^n,\,d\sigma)}.
\end{split}
\end{equation}
Hence, by induction, 
\[
\Vert u_{j+1} \Vert_{L^{\gamma+q}(\mathbb{R}^n,\,d\sigma)} 
\leq c^{\frac{p-1}{p-1-q}} < +\infty, \quad j=0, 1, \ldots .
\]
Finally, applying the Monotone Convergence Theorem to $\lbrace u_{j} \rbrace_{j=0}^{\infty}$, we see that the pointwise limit $u=\lim_{j \rightarrow \infty} u_j$ exists so that $u>0$, 
$u \in L^{\gamma+q}(\mathbb{R}^n, d\sigma)$, and $u$ satisfies equation \eqref{eq:int-wolff}.

The necessity part follows immediately from the global pointwise lower bound \eqref{inq:lowerbound} of supersolutions to \eqref{eq:int-wolff}.
\end{proof}

Let us consider the following weighted norm inequality with Lebesgue measure on the left-hand side:
\begin{equation}\label{winq:wolff-dx-sigma} 
\Vert \W_{\alpha, p}(f d\sigma)\Vert_{L^r(\R^{n})} \leq C \,  \Vert f \Vert^{\frac{1}{p-1}}_{L^{s} (\R^{n},\, d\sigma)}, \quad \forall f \in L^{s} (\R^{n}, d \sigma), 
\end{equation} 
where $r>p-1$, $s>1$, and $0<\alpha  p< n$. 

It is well-known (see \cite{HJ}, \cite{JPW}) that 
\begin{equation}\label{mw} 
\Vert \W_{\alpha, p}(f d\sigma)\Vert_{L^r(\R^{n})}  \approx  \Vert \I_{\alpha p}(f d\sigma)\Vert^{\frac{1}{p-1}}_{L^{\frac{r}{p-1}}(\R^{n})}, 
\end{equation}
with the constants of equivalence that depend on $\alpha, p, q, r, n$, 
but not on $f$ and $\sigma$.

Hence, for $r>p-1$ and $s>1$, 
 \eqref{winq:wolff-dx-sigma} is   equivalent by duality to
\begin{equation}\label{winq:wolff-dx-sigma-dual} 
\Vert  \I_{\alpha p} (g dx)\Vert_{L^{s'}(\R^{n},\,d\sigma)} 
\leq c \, C^{p-1}\, \Vert g \Vert_{L^{\frac{r}{r-p+1}} (\R^{n})}, \quad \forall g \in L^{\frac{r}{r-p+1}} (\R^{n}), 
\end{equation} 
where $(\frac{r}{p-1})'=\frac{r}{r-p+1}$, $s'=\frac{s}{s-1}$ are the dual exponents, and $c=c(\alpha, p, r, s, n)$ 
is a positive constant. 

The next lemma gives a sufficient condition for the validity of \eqref{winq:wolff-dx-sigma}.

\begin{Lem}\label{lem1-wolff}  
Let $1< p< \infty$, $0<\alpha< \frac{n}{p}$, $0<q<p-1$, and let  $\sigma \in \M^{+}(\R^{n})$. Suppose that $\frac{n(p-1)}{n-\alpha p}<r<\infty$
and 
\begin{equation}\label{cond:dx-wolff} 
\W_{\alpha, p} \sigma \in L^{\frac{r(p-1)}{p-1-q}}(\R^n).
\end{equation} 
Then \eqref{winq:wolff-dx-sigma} is valid with $s= \frac{r(n-\alpha p)- (p-1) n}{nq}+1$.
Moreover, there exists a positive constant $c=c(\alpha, p, q, r, n)$ such that
\begin{equation}\label{winq:dx-wolff-2} 
\Vert \W_{\alpha, p}(f d\sigma)\Vert_{L^r(\R^{n})} 
\leq c  \Vert \W_{\alpha, p}\sigma \Vert^{\frac{1}{s'}}_{L^{\frac{r(p-1)}{p-1-q}} (\R^{n})}  
\Vert f \Vert^{\frac{1}{p-1}}_{L^{s} (\R^{n}, d \sigma)}, 
\end{equation} 
for all $ f \in L^{s} (\R^{n}, d \sigma)$.
\end{Lem}

\begin{proof} 
As was shown above,  \eqref{winq:dx-wolff-2} is equivalent to 
\begin{equation}\label{dual-2} 
\Vert \I_{\alpha p}(g\,dx)\Vert_{L^{s'}(\R^{n},\,d\sigma)} 
\leq c  \Vert \W_{\alpha, p}\sigma \Vert^{\frac{p-1}{s'}}_{L^{\frac{r(p-1)}{p-1-q}} (\R^{n})}
\Vert g\Vert_{L^{\frac{r}{r-p+1}} (\R^{n})}, 
\end{equation} 
for all $g \in L^{\frac{r}{r-p+1}} (\R^n)$, where
without loss of generality we may assume that $g \ge 0$.

To prove the preceding inequality, we 
apply the iterated inequality \eqref{iterated1-wolff} in the special case $p=2$, to 
the Riesz potential  $\I_{\alpha p} (g \, d x)$, 
and $t=s'$, so that 
\[
\Big( \I_{\alpha p} (g \, d x) \Big)^{s'} \le c \, \I_{\alpha p} \left [ \Big( \I_{\alpha p} (g \, d x) \Big)^{s'-1} \, g\, dx \right]. 
\]
Hence, by the preceding inequality and Fubini's theorem, 
\[
\begin{split}
\int_{\R^{n}} \Big( \I_{\alpha p} (g \, d x) \Big)^{s'}\,d\sigma 
& \leq c \int_{\R^{n}}  \I_{\alpha p} \Big[ \Big( \I_{\alpha p} (g \, d x) \Big)^{s'-1} \, g\, dx 
  \Big] d \sigma 
  \\ 
&= c \int_{\R^{n}} \Big(\I_{\alpha p} (g \, d x) \Big)^{s'-1} (\I_{\alpha p} \sigma) \, g\,dx.  
\end{split}
\]
By H\"older's inequality with the three exponents  
$\frac{r}{q}$, 
$\frac{r}{p-1-q}>1$ and 
$\frac{r}{r-p+1}$, we estimate
\[
\begin{split}
 \int_{\R^{n}} \Big(\I_{\alpha p} (g \, d x) \Big)^{s'-1} (\I_{\alpha p} \sigma) \, g\,dx
& \le \big\| \I_{\alpha p}(g dx) \big\|^{s'-1}_{L^{\frac{r(s'-1)}{q}}(\R^n)} 
\\ 
& \times\Vert \I_{\alpha p} \sigma\Vert_{L^\frac{r}{p-1-q}(\R^n)}  \Vert g \Vert_{L^\frac{r}{r-p+1}(\R^n)}. 
\end{split}
\]
We observe that
$\frac{r(s'-1)}{q} = \frac{rn}{r(n-\alpha p)- (p-1) n}>1$ and
\[
\tfrac{q}{r(s'-1)} = \tfrac{r(n-\alpha p) -(p-1)n}{rn} = \tfrac{r-p+1}{r} - \tfrac{\alpha p }{n}.
\] 
Hence, by 
 Sobolev's inequality for the Riesz potential  $\I_{\alpha p}$,  
\[
\big\| \I_{\alpha p}(g dx) \big\|_{L^{\frac{r(s'-1)}{q}}(\R^n)}
\leq c \big\| g \big\|_{L^{\frac{r}{r-p+1}}(\R^n)}.
\]
Combining the above estimates proves 
\begin{equation*}\label{inq:weight-wolff-dual-2} 
\Vert \I_{\alpha p}(g\,dx)\Vert^{{s'}}_{L^{s'}(\R^{n},\,d\sigma)} 
\leq c  \Vert \I_{\alpha p}\sigma \Vert_{L^{\frac{r}{p-1-q}} (\R^{n})}
\Vert g\Vert^{s'}_{L^{\frac{r}{r-p+1}} (\R^{n})}.  
\end{equation*} 
Notice that  by \eqref{mw} with $f\equiv 1$ and $\frac{r(p-1)}{p-1-q}$ in place of $r$, we have 
\begin{equation*}
 \Vert \I_{\alpha p}\sigma \Vert_{L^{\frac{r}{p-1-q}} (\R^{n})} 
 \approx \Vert \W_{\alpha, p}\sigma \Vert^{p-1}_{L^{\frac{r(p-1)}{p-1-q}} (\R^{n})}.  
 \end{equation*}
This completes the proof of  \eqref{dual-2}, and hence 
\eqref{winq:dx-wolff-2}.
\end{proof}

The following lemma provides a crucial link between conditions 
\eqref{cond:sigma-gamma} and \eqref{cond:dx-wolff}. 
In particular, it shows that \eqref{cond:sigma-gamma} yields both weighted norm inequalities \eqref{winq:wolff-sigma-sigma} and \eqref{winq:wolff-dx-sigma} 
for suitable $r$ and $s$, in light of Lemma \ref{lem:wolff} and Lemma \ref{lem1-wolff}.

\begin{Lem}\label{lemma-link-p}  
Let $1< p< \infty$, $0<\alpha< \frac{n}{p}$, $0<q<p-1$, and let  $\sigma \in \M^{+}(\R^{n})$. 
If $\frac{n(p-1)}{n-\alpha p}<r<\infty$, then  
\eqref{cond:sigma-gamma} with $\gamma = \frac{r(n-\alpha p)-(p-1)n}{n}$ 
implies \eqref{cond:dx-wolff}.
\end{Lem}

\begin{proof}
For  $k \in \N$, let $\sigma_{k}:= \chi_{\Omega_{k}}\sigma$, where $\chi_{\Omega_{k}}$ is the characteristic 
function of the set 
\[
\Omega_{k} =\{ x \in \R^{n}: \,\, \W_{\alpha, p}\sigma(x) \leq k
\}\cap B(0, k). 
\]
Then $\W_{\alpha, p}\sigma_k (x)\le k$ on  
the support of $\sigma_k$, and  by the weak maximum principle 
\eqref{wolff-max}, we have 
$\W_{\alpha, p}\sigma_k (x)\le  c(\alpha, p, n) \, k$ for all $x \in \R^n$. 

Moreover, if $|x|\ge 2 k$, then $\sigma_k(B(x, \rho))=0$ for $0<\rho<\frac{|x|}{2}$, and  
hence, for a positive constant $c=c(\alpha, p, n)$, 
\[
\W_{\alpha, p}\sigma_k (x)\le \int_{\frac{|x|}{2}}^\infty \left[\frac{\sigma_k(B(x, \rho))}
{\rho^{n-\alpha p}}
\right]^{\frac{1}{p-1}} \frac{d \rho}{\rho}\le  \, c \, 
\sigma(B(0, k))^{\frac{1}{p-1}} 
|x|^{-\frac{n-\alpha p}{p-1}}.
\]
Thus, $
\W_{\alpha, p}\sigma_k (x) \leq C (|x|+1)^{-\frac{n-\alpha p}{p-1}}$, for all $x \in \R^n$. 
Since $r>\frac{n(p-1)}{n-\alpha p}$ and $0<q<p-1$, it follows that 
\[
\int_{\R^n} (\W_{\alpha, p}\sigma_{k})^{\frac{r(p-1)}{p-1-q}} dx < + \infty.
\]

Applying the iterated inequality \eqref{iterated1-wolff} with $t=\frac{p-1}{p-1-q}>1$, along with the 
weighted norm inequality \eqref{winq:dx-wolff-2} with 
$f=(\W_{\alpha, p}\sigma_{k})^{\frac{q(p-1)}{p-1-q}}$, we deduce
\begin{align*} 
& \int_{\R^n} (\W_{\alpha, p}\sigma_{k}) ^{\frac{r(p-1)}{p-1-q}}\,dx 
\leq c \int_{\R^n} \left[ \W_{\alpha, p} (f d \sigma_{k}) \right]^r dx\\ 
&\leq  c \big\| \W_{\alpha, p}\sigma_{k} \big\|^{\frac{r}{s'}}_{L^{\frac{r(p-1)}{p-1-q}} (\R^n)} 
             \big\| f \big\|^{\frac{r}{p-1}}_{L^{s} (\R^n,\,d\sigma_{k})} \\
&\leq c \Big[ \int_{\R^n} (\W_{\alpha, p}\sigma_{k})^{\frac{r(p-1)}{p-1-q}}\,dx \Big]^{\frac{p-1-q}{(p-1)s'}}
\Big[ \int_{\R^n} (\mathbf{W}_{\alpha, p} \sigma)^{\frac{qs(p-1)}{p-1-q}} d \sigma\Big]^{\frac{r}{s(p-1)}}, 
\end{align*}
where $s= \frac{r(n-\alpha p)- (p-1) n}{nq}+1$. In other words, 
\[
\Big[ \int_{\R^n} (\W_{\alpha, p}\sigma_{k})^{\frac{r(p-1)}{p-1-q}}\,dx \Big]^{1-\frac{p-1-q}{s'(p-1)}}
\leq c \Big[ \int_{\R^{n}} (\W_{\alpha, p} \sigma)^{\frac{qs(p-1)}{p-1-q}} d \sigma\Big]^{\frac{r}{s(p-1)}}, 
\]
where the right-hand side integral is finite by assumption \eqref{cond:sigma-gamma},  since
$\frac{qs(p-1)}{p-1-q}=\frac{(\gamma+q)(p-1)}{p-1-q}$. Notice that $1-\frac{p-1-q}{(p-1)s'}>0$.
Passing to the limit as $k \rightarrow \infty$ in the preceding estimate yields \eqref{cond:dx-wolff}.
\end{proof}

We now prove Theorem \ref{thm:main-p-laplacian} and Theorem \ref{thm:main2}.

\begin{proof}[Proof of Theorem \ref{thm:main-p-laplacian}] It is well known 
\cite{HKM} that,  
for $p \ge n$,  there is no positive $p$-superharmonic functions on $\R^n$. Hence, 
\eqref{eq:p-laplacian}  has no nontrivial solutions. 

In the case $1<p<n$,  suppose \eqref{eq:p-laplacian}  has a nontrivial solution 
$u \in L^r(\R^n)$, $r>0$. Let $d \omega =u^q d\sigma$,  and 
pick $R>0$ so that $\omega(B(0, R))>0$. Then by Theorem \ref{pointwise_est_p-superharmonic}, we have, for $|x|> 2R$, 
\[
u(x) \ge C \,  \W_{1, p}\omega (x)\ge C \, \omega(B(0, R))^{\frac{1}{p-1}} \, |x|^{-\frac{n-p}{p-1}}.
\]
Hence,  $u \in L^r(\R^n)$ yields   
$r>\frac{n(p-1)}{n-p}$.  

Suppose now that $1<p<n$, $r>\frac{n(p-1)}{n-p}$, and 
\eqref{cond:dsigma-wolff} is valid. In view of Theorem \ref{thm:wolff-existence} with $\alpha=1$, there exists a positive solution 
$v \in L^{\gamma +q}(\Omega, d\sigma)$, where $\gamma = \frac{r(n-p)-n(p-1)}{n}>0$, to the integral equation
\begin{equation}\label{eq:int-wolff-1}
v = \W_{1,p}(v^{q}d\sigma) \quad \text{in} \;\; \R^n.
\end{equation}
Moreover, by Lemma \ref{lem1-wolff} and Lemma \ref{lemma-link-p}, the weighted
norm inequality \eqref{winq:wolff-dx-sigma} holds with $s=\frac{\gamma+q}{q}$.
Letting $f=v^{q} \in L^{\frac{\gamma + q}{q}}(\R^n)$ in \eqref{winq:wolff-dx-sigma} yields
\[
\begin{split}
\| v \|_{L^{r}(\R^n)} &= \big\| \W_{1, p}(v^{q}d\sigma) \big\|_{L^{r}(\R^n)}
\leq c \| v^q \|^{\frac{1}{p-1}}_{L^{\frac{\gamma + q}{q}}(\R^n, d \sigma)} \\
&= c \| v \|_{L^{\gamma +q}(\R^n, d \sigma)}^{\frac{q}{p-1}}
< +\infty.
\end{split}
\]
This shows that the solution $v \in L^{r}(\R^n)$ as well.

Using an iterative procedure  as in the proof of \cite[Theorem 1.1]{CV2}, 
based on  Theorem \ref{weak_cont_p-Laplacian} and Theorem \ref{pointwise_est_p-superharmonic},  
we construct a (minimal) positive $p$-superharmonic solution $u$
to \eqref{eq:p-laplacian} so that $u(x) \le C \, v(x)$. Consequently, 
$u \in L^{r}(\R^n)$.
\end{proof}

\begin{proof}[Proof of Theorem \ref{thm:main2}] 
Since a superharmonic solution $u$ to \eqref{eq:frac-laplacian} is understood in the sense that
\begin{equation}\label{eq:int-riesz}
u = \I_{2\alpha}(u^{q}d\sigma) \quad \text{in} \;\; \R^n,
\end{equation}
then by using an argument similar to that  in the proof of Theorem \ref{thm:main-p-laplacian} with 
$\I_{2\alpha} = \W_{\alpha, 2}$ in place of $\W_{1,p}$, we see that condition \eqref{cond:dsigma-riesz} yields the existence of a (minimal) positive solution $u \in L^{r}(\R^n)$
to \eqref{eq:frac-laplacian} for $r>\frac{n}{n- 2 \alpha}$. If $0<r \le \frac{n}{n- 2 \alpha}$, then there is no nontrivial solution $u \in L^r(\R^n)$, since 
$u(x)=\I_{2\alpha}(u^{q}d\sigma)(x)\ge C \, |x|^{2 \alpha-n}$ for $|x|$ large, as in   the 
$p$-Laplace case considered above. 
\end{proof}

The next proposition is used to deduce Corollaries \ref{cor:main1} and \ref{cor:main3}. 

\begin{Prop}\label{prop:wolff}
Let $1<p<n$, $0<\alpha<\frac{n}{p}$ and $\beta>0$.
If $\sigma \in L^{s}(\R^n)$ is a nonnegative function, where $s=\frac{n(\beta+p-1)}{n(p-1)+p\alpha\beta}$, then
\begin{equation}\label{cond:sigma-gen}
\W_{\alpha, p}\sigma \in L^{\beta}(\R^{n}, d\sigma).
\end{equation}
\end{Prop}

\begin{proof} 
Since $s>1$, by H\"{o}lder's inequality, we have 
\[
\int_{\R^n} (\W_{\alpha, p}\sigma)^{\beta}\,d\sigma
\leq \big\| \W_{\alpha, p}\sigma \big\|_{L^{\beta s'}(\R^n)}^{\beta}
\| \sigma \|_{L^{s}(\R^n)}.
\]
Notice that
$
\frac{p-1}{\beta s'} = \frac{1}{s} - \frac{\alpha p}{n}>0.
$
Estimating again   the norm of $\W_{\alpha, p} \sigma$ in terms of the norm of 
the Riesz potential $\I_{\alpha p} \sigma$ by \eqref{mw}, along with Sobolev's inequality 
for  $\I_{\alpha p} \sigma$, we  obtain 
\[
\big\| \W_{\alpha, p}\sigma \big\|_{L^{\beta s'}(\R^n)}
\leq c \big\| \I_{\alpha p}\sigma \big\|_{L^{\frac{\beta s'}{p-1}}(\R^n)}^{\frac{1}{p-1}}
\leq c \| \sigma \|_{L^{s}(\R^n)}^{\frac{1}{p-1}}.
\]
Combining the preceding estimates yields \eqref{cond:sigma-gen}.
\end{proof}

Letting  $\beta=\frac{(\gamma+q)(p-1)}{p-1-q}$, where  $0<q<p-1$, and 
$\gamma=\frac{r(n-\alpha p)-n(p-1)}{n}$, where $r>\frac{n(p-1)}{n-\alpha p}$, 
we obtain $\beta>0$ and $s=\frac{nr}{n(p-1-q)+ \alpha p r}>1$ in Proposition \ref{prop:wolff}. Notice that condition \eqref{cond:sigma-gen} coincides with \eqref{cond:sigma-gamma}. 
Consequently, $\sigma \in L^{s}(\R^n)$ is sufficient for  \eqref{cond:sigma-gamma} 
by Proposition \ref{prop:wolff}. 

In the case $\alpha=1$, we obtain that 
\eqref{cond:sigma-suff} implies \eqref{cond:dsigma-wolff}.
Thus,  Corollary \ref{cor:main1} follows from Theorem \ref{thm:main-p-laplacian}.

Similarly, applying Proposition \ref{prop:wolff} with $p=2$, we see that 
\eqref{cond:sigma-suff2} implies \eqref{cond:dsigma-riesz}. 
Hence, Corollary \ref{cor:main3} holds in view of Theorem \ref{thm:main2}.

\section{Solutions to problem \eqref{eq:laplacian}}
As mentioned above, our method of proof of Theorem \ref{thm:main3} 
is based on both the  
weighted norm inequality \eqref{weighted_norm_ineq}, and a similar inequality  
with Lebesgue measure on the left-hand side:
\begin{equation}\label{inq:weighted} 
\Vert \mathbf{G} (f d\sigma)\Vert_{L^r(\Omega)} \leq c \Vert f \Vert_{L^{s} (\Omega,\, d\sigma)}, \quad \forall f \in L^{s} (\Omega, d \sigma),
\end{equation} 
for $1<r,s<\infty$. Notice that by duality, \eqref{inq:weighted} is equivalent to
\begin{equation}\label{inq:weight-dual} 
\Vert \mathbf{G} (g dx)\Vert_{L^{s'}(\Omega,\,d\sigma)} 
\leq c \Vert g \Vert_{L^{r'} (\Omega)}, \quad \forall g \in L^{r'} (\Omega),
\end{equation} 
where $r'=\frac{r}{r-1}$ and $s'=\frac{s}{s-1}$.

We observe that such inequalities were characterized when $\g$ is
the Riesz potential on $\R^n$, in terms of Riesz capacities, by Maz'ya and Netrusov,
and in terms of nonlinear  potentials,  by Cascante, Ortega, and Verbitsky (see \cite{COV}, \cite{Maz}).
Our first lemma gives sufficient conditions for the validity of
 \eqref{inq:weighted}.

\begin{Lem}\label{lemma1}  
Let $0<q<1$ and $\sigma \in \M^{+}(\Omega)$. Let $G$ be a positive lower semicontinuous kernel on 
$\Omega \times \Omega$, which satisfies the WMP and \eqref{cond:Riesz0} for some 
$0<\alpha<\frac{n}{2}$. Suppose that $\frac{n}{n-2\alpha}<r<\infty$
and that condition \eqref{cond:dx-green} holds.
Then \eqref{inq:weighted} is valid with $s= \frac{r(n-2\alpha) - n(1-q)}{nq}$.
Moreover,
\begin{equation}\label{inq:weighted-2} 
\Vert \mathbf{G}(f d\sigma)\Vert_{L^r(\Omega)} 
\leq c  \Vert \mathbf{G} \sigma \Vert^{\frac{1}{s'}}_{L^{\frac{r}{1-q}} (\Omega)}  
\Vert f \Vert_{L^{s} (\Omega, d \sigma)}, \quad \forall f \in L^{s} (\Omega, d \sigma),
\end{equation} 
where $c$ is a positive constant independent of $f$ and $\sigma$.
\end{Lem}

\begin{proof} 
Observe that the hypotheses ensure
\[
1<s<\infty, \quad \tfrac{(s'-1)r}{q}>1 \quad  \text{and} \quad  
\tfrac{1}{r'} - \tfrac{q}{(s'-1)r} = \tfrac{2\alpha}{n}.
\]
By duality, \eqref{inq:weighted-2}  is equivalent to 
\begin{equation}\label{inq:weight-dual-2} 
\Vert \mathbf{G} (g\,dx)\Vert_{L^{s'}(\Omega,\,d\sigma)} 
\leq c  \Vert \mathbf{G} \sigma \Vert^{\frac{1}{s'}}_{L^{\frac{r}{1-q}} (\Omega)}
\Vert g \Vert_{L^{r'} (\Omega)}, \quad \forall g \in L^{r'} (\Omega),
\end{equation} 
where $c$ is a positive constant independent of $g$  and $\sigma$. 
Without loss of generality, suppose $g \in L^{r'} (\Omega)$ with $g \ge 0$.
Applying iterated inequality \eqref{iterated1} 
with $t=s'$, followed by Fubini's theorem and H\"older's inequality with exponents $\frac{r}{1-q}$ and 
$\frac{r}{r-1+q}$, we deduce
\[
\begin{split}
 \int_\Omega \Big(\mathbf{G} (g \, d x) \Big)^{s'}\,d\sigma 
&  \leq c \int_\Omega \mathbf{G} \Big[ \Big(\mathbf{G} (g \, d x) \Big)^{s'-1} \, g dx 
  \Big] d \sigma \\ 
&= c \int_\Omega \Big(\mathbf{G} (g \, d x) \Big)^{s'-1} (\mathbf{G} \sigma) \, g dx.
\end{split}
\]
Using  H\"older's inequality with the exponents $\frac{r}{q}$, 
$\frac{r}{1-q}$ and $r'$, we estimate
\[
\begin{split}
\int_\Omega \Big(\mathbf{G} (g \, d x) \Big)^{s'-1} (\mathbf{G} \sigma) \, g dx
& \leq \Vert \mathbf{G} (g \, d x) \Vert^{s'-1}_{L^{\frac{(s'-1)r}{q}} (\Omega)} \\ & 
 \times \Vert \mathbf{G} \sigma \Vert_{L^\frac{r}{1-q}(\Omega)}  
\Vert g \Vert_{L^{r'}(\Omega)} .
\end{split}
\]
Denote by $\tilde{g}$ the zero extension of $g$ to $\R^n$. By assumption 
\eqref{cond:Riesz0} and Sobolev's inequality for Riesz potentials, we have
\[
\| \g(g\,dx) \|_{L^{\frac{(s'-1)r}{q}}(\Omega)}
\leq c \| \I_{2\alpha}\tilde{g}\|_{L^{\frac{(s'-1)r}{q}}(\R^n)} 
\leq  c \| \tilde{g}\|_{L^{r'}(\R^n)} 
= c \| g \|_{L^{r'}(\Omega)}.
\]
Combining the preceding inequalities yields \eqref{inq:weight-dual-2} as desired.
\end{proof}

The following lemma provides an important link between conditions \eqref{cond:dx-green} and \eqref{cond:dsigma-green}. 
In particular, it shows that \eqref{cond:dsigma-green} yields both weighted norm inequalities \eqref{weighted_norm_ineq} and \eqref{inq:weighted} for suitable $r$ and $s$, in light of Theorem \ref{thm_Ver17} and Lemma \ref{lemma1}.

\begin{Lem}\label{lemma2}  Let $0<q<1$ and $\sigma \in \M^{+}(\Omega)$ with $\sigma \not\equiv 0$. Suppose $G$ is a positive lower semicontinuous kernel on 
$\Omega \times \Omega$, which satisfies the WMP and \eqref{cond:Riesz0} for some 
$0<\alpha<\frac{n}{2}$. If $\frac{n}{n-2\alpha}<r<\infty$, then  
\eqref{cond:dsigma-green} with $\gamma=\frac{r(n-2 \alpha)-n}{n}$ 
implies \eqref{cond:dx-green}.
\end{Lem}

\begin{proof}
Let $\lbrace \Omega_{k}\rbrace_{k \in \N}$ be an  increasing exhaustive sequence of relatively compact open sets in $\Omega $ such that $\g\sigma (x)\leq k$ for $x\in \Omega_k$. Set $\sigma_{k}:= \chi_{\Omega_{k}}\sigma$. By the WMP, $\g\sigma_k \leq h k$ for all 
$x \in \Omega$. Moreover, if $\Omega_{k}\subset B(0, R)$, and $|x|>2R$, then 
$\g\sigma_k (x)\leq C \, |x|^{2 \alpha-n}$ by \eqref{cond:Riesz0}. Consequently, $\g\sigma_k (x)
\leq C \, (1+|x|)^{2 \alpha-n}$ for all $x\in\Omega$. Since $r>\frac{n}{n-2 \alpha}$ and 
$0<q<1$, we have 
\[
\int_{\Omega} (\g\sigma_{k})^{\frac{r}{1-q}} dx < + \infty.
\]
Applying the iterated inequality \eqref{iterated1} with $t=\frac{1}{1-q}$, followed by 
the weighted norm inequality \eqref{inq:weighted-2} with $f=(\g\sigma_{k})^{\frac{q}{1-q}}$, we estimate
\begin{align*} 
& \int_\Omega (\g\sigma_{k})^{\frac{r}{1-q}}\,dx 
\leq c \int_\Omega \left[ \mathbf{G} (f d \sigma_{k}) \right]^r dx\\ 
&\leq  c \big\| \g\sigma_{k} \big\|^{\frac{r}{s'}}_{L^{\frac{r}{1-q}} (\Omega)} 
             \big\| f \big\|^r_{L^{s} (\Omega,\,d\sigma_{k})}\\
&\leq c \Big[ \int_{\Omega} (\g\sigma_{k})^{\frac{r}{1-q}} dx \Big]^{\frac{1-q}{s'}}
\Big[ \int_\Omega (\mathbf{G} \sigma)^{\frac{qs}{1-q}} d \sigma\Big]^{\frac{r}{s}},
\end{align*}
where $s= \frac{r(n-2\alpha) - n(1-q)}{nq}$. In other words, 
\[
\Big[ \int_{\Omega} (\g\sigma_{k})^{\frac{r}{1-q}} dx\Big]^{1-\frac{1-q}{s'}}
\leq c \Big[ \int_\Omega (\mathbf{G} \sigma)^{\frac{qs}{1-q}} d \sigma\Big]^{\frac{r}{s}}
< +\infty,
\]
where the right-hand side integral is finite by assumption \eqref{cond:dsigma-green} with $\gamma=\frac{r(n-2 \alpha)-n}{n}$,  since
$qs=\gamma+q$. Here obviously $1-\frac{1-q}{s'}>0$.
Passing to the limit as  $k \rightarrow \infty$
in the preceding estimate yields \eqref{cond:dx-green}.
\end{proof}
We are now ready to prove Theorem \ref{thm:main3}.
\begin{proof}[Proof of Theorem \ref{thm:main3}]

Suppose \eqref{cond:dsigma-green} holds.
In view of Theorem \ref{thm_Ver17}, there exists a positive solution 
$u \in L^{\gamma +q}(\Omega, d\sigma)$ to the integral  equation
\begin{equation}\label{eq:int-green}
u = \g(u^{q}d\sigma) \quad \text{in} \;\; \R^n.
\end{equation}
Further, in light of Lemma \ref{lemma1} and Lemma \ref{lemma2}, 
the weighted norm inequality \eqref{inq:weighted} holds with $s=\frac{\gamma + q}{q}$. 
Hence, applying \eqref{inq:weighted} with $f =u^{q} \in L^{\frac{\gamma+q}{q}}(\Omega, d\sigma)$ yields
\[
\| u \|_{L^{r}(\Omega)} 
\leq \| \g(u^{q}d\sigma) \|_{L^{r}(\Omega)}
\leq c \| u^{q} \|_{L^{\frac{\gamma + q}{q}}(\Omega)}
= c \| u \|_{L^{\gamma +q}(\Omega)}^{q}
< + \infty.
\]
This shows that  $u\in L^{r}(\Omega)$ is a positive superharmonic solution to \eqref{eq:laplacian}.\end{proof}

To deduce Corollary \ref{cor:main3}, we need the following proposition 
(\cite[Proposition 4.8]{SV2}).  

\begin{Prop}\label{prop3} Let $\Omega \subseteq \R^n$. 
Let $G$ be a positive kernel on $\Omega\times \Omega$ which satisfies \eqref{cond:Riesz0} with $0<2 \alpha<n$. 
 If $\sigma \in L^{s}(\Omega)$ is a positive function, where 
$
s=\frac{n(\beta+1)}{n+2\alpha\beta},
$ and $\beta>0$, 
then
\begin{equation}\label{cond_general}
\g\sigma \in L^{\beta}(\Omega, d\sigma).
\end{equation}
\end{Prop}

Proposition \ref{prop3} with $\alpha=1$ and 
$\beta=\frac{r(n-2)-n(1-q)}{n(1-q)}=\frac{\gamma+q}{1-q}$ 
shows that  \eqref{cond:sigma-suff3} implies \eqref{cond:dsigma-green}, and 
consequently  Theorem \ref{thm:main3} yields Corollary \ref{cor:main3}.

\end{document}